\documentclass[a4paper]{amsart}
\usepackage{amssymb}
\usepackage{amsmath}
\newcommand{\f}{\varphi}
\newcommand\G{\mathrm{I}\!\mathrm{\Gamma}}
\renewcommand\L{\mathcal L}
\newcommand\K{\mathcal K}
\newcommand\U{\mathcal U}

\newcommand\B{\mathcal B}

\newcommand{\s}{\mathfrak{s}}

\newcommand{\skp}[2]{\left<#1,#2\right>}

\newcommand\ind{\mathop{\mathrm{ind}}\nolimits}

\theoremstyle{plain}
\newtheorem{theorem}{Theorem}[section]

\newtheorem{proposition}[theorem]{Proposition}

\theoremstyle{definition}
\newtheorem{definition}{Definition}[section]

\theoremstyle{remark}
\newtheorem{remark}{Remark}[section]

\begin{document}



\title[The index of a subspatial product system]{The index of a subspatial product system over a Hilbert $C^*$-module -- an example}

\author{Dragoljub J.\ Ke\v cki\'c}

\address{Faculty of Mathematics\\ University of Belgrade\\ Student\/ski trg 16-18\\
11000 Belgrade, Serbia}

\email{keckic@matf.bg.ac.rs}

\author{Biljana Vujo\v sevi\'c}

\address{Faculty of Mathematics\\ University of Belgrade\\ Student\/ski trg 16-18\\
11000 Belgrade, Serbia}

\email{bvujosevic@matf.bg.ac.rs}



\begin{abstract}
Recently, Bhat, Liebscher and Skeide exhibited an unusual example of product system of Hilbert $C^*$ modules.
Namely, a subsystem of a Fock system, that is not Fock. Later, the authors defined the index of any product
system of Hilbert $C^*$-modules, generalizing earlier partial definitions. In this note we compute the index
of the mentioned interesting subsystem, and also we determine all its units.%
\end{abstract}

\maketitle

\section{Introduction}\label{secdefind}

Product systems over $\mathbb C$ have been studied during last several decades in connection with
$E_0$-semigroups acting on a type $I$ factor. Although the main problem of classification of all non
isomorphic product systems is still open, this theory is well developed. The reader is referred to Arveson's
book \cite{Arv03} and references therein. In the present century there are some significant results that
generalizes this theory to product systems over some $C^*$-algebra $\mathcal B$. The following definitions
can be found, for instance, in \cite{BS00}, \cite{SK02}, \cite{JFA04}.

\begin{definition} a) Product system over  $C^*$-algebra $\mathcal B$ is a family $(E_t)_{t\ge0}$ of Hilbert $\mathcal
B-\mathcal B$ modules, with $E_0\cong\mathcal B$, and a family of (unitary) isomorphisms
$$\f_{t,s}:E_t\otimes E_s\to E_{t+s},$$
where $\otimes$ stands for the so called inner tensor product obtained by identifications $u b\otimes v\sim
u\otimes bv$, $u\otimes vb\sim(u\otimes v)b$, $bu\otimes v\sim b(u\otimes v)$, ($u\in E_t$, $v\in E_s$,
$b\in\mathcal B$) and then completing in the inner product $\skp{u\otimes v}{u_1\otimes v_1}=\skp v{\skp
u{u_1}v_1}$;

b) Unit on $E$ is a family $u_t\in E_t$, $t\ge0$, such that $u_0=1$ and $\f_{t,s}(u_t\otimes u_s)=u_{t+s}$,
which will be abbreviated to $u_t\otimes u_s=u_{t+s}$. A unit $u_t$ is unital if $\skp{u_t}{u_t}=1$. It is
central if for all $b\in\mathcal B$ and all $t\ge0$ there holds $bu_t=u_tb$;

d) A product system $E$ is called spatial if it admits a central unital unit.
\end{definition}

Product systems over $\mathbb C$ are special case of the previous definition, and we shall refer to them as
Arveson systems.

Note that this definition does not include any technical condition such as measurability, continuity etc.\ of
product system. In fact, it is customary to pose such conditions directly on units.

\begin{definition} Two units $u_t$ and $v_t$ give rise to the family of mappings $\K^{u,v}_t:\mathcal B\to\mathcal B$,
given by $\K^{u,v}_t(b)= \skp{u_t}{bv_t}.$ All $\K^{u,v}_t$ are bounded $\mathbb C$-linear operators on
$\mathcal B$, and this family forms a semigroup. The set of units $S$ is continuous if the corresponding
semigroup $(K_t^{\xi,\eta})_{\xi,\eta\in S}$ (with respect to Schur multiplying) is uniformly continuous. A
single unit $u_t$ is uniformly continuous, or briefly just continuous, if the set $\{u\}$ is continuous, that
is, the corresponding family $\K^{u,u}_t$ is continuous in the norm of the space $B(\mathcal B)$ (algebra of
all bounded $\mathbb{C}$-linear operators on $\B$).
\end{definition}

As it can be seen in \cite{JFA04}, for a (uniformly) continuous set of units $\mathcal U$, there can be
formed a uniformly continuous completely positive definite semigroup ($CPD$-semigroup in further)
$\mathcal{K}=(\mathcal{K}_t)_{t\in{\mathbb{R}}_{+}}$.

Denote by $\mathcal{L}=\frac{d}{dt}\mathcal{K} \mid _{t=0}$ the generator of CPD-semigroup $\mathcal{K}$. It
is well known \cite{JFA04} that $\mathcal{L}$ is conditionally completely positive definite and also holds
\begin{equation}\label{L*}
\mathcal L^{y,x}(b)=\mathcal L^{x,y}(b^*)^*,\ x,y \in \U,\ b \in \B.
\end{equation}

It is known that $\K$ is uniquely determined by $\mathcal L$. More precisely, $\K$ can be recovered from
$\mathcal L$ by $\K=e^{t\mathcal L}$ using Schur product, i.e.
\begin{equation}\label{e^L}
\K^{x,y}_t(b)=\skp{x_t}{by_t}=(\exp t\mathcal L^{x,y})(b).
\end{equation}

\begin{remark}It should distinct the continuous set of units and the set of continuous units. In the second case
only $\K^{\xi,\xi}_t$ should be uniformly continuous for $\xi\in S$, whereas in the first case all
$\K_t^{\xi,\eta}$ should be uniformly continuous.
\end{remark}

Analogously to Arveson systems, product systems over $\B$ can be classified in terms of how many units it
has. Namely

\begin{definition} A product system $E$ is of type $I$ if it is generated by some continuous set of units $\mathcal
U$, i.e.\ if for all $t>0$ the $\B$-linear span of $\{u_t\:|\:u\in\mathcal U\}$ is dense in $E_t$ for some
continuous set of units $\mathcal U$. It is said to be of type $II$ if it has at least one continuous unit
and it is not of type $I$. Otherwise, it is of type $III$.
\end{definition}

In Arveson case, type $I$ systems are completely determined by its index, a positive integer obtained as the
dimension of a suitable linear space constructed on the base of the set of units, whereas the situation is
more complicated for non type $I$ systems.

Among all product systems, so called time ordered Fock modules (or timed ordered product systems) have an
important role. It can be constructed as follows.

Let $F$ be a Hilbert $\B-\B$ module. By $L^2(\mathbb{R}_+,F)$ we denote the completion of the exterior tensor
product $F \otimes L^2\mathbb(R_+)$. Then $L^2(\mathbb{R}_+,F)$ is a Hilbert $\B-\B$ module with obvious
structure. As usual we have $L^2(\mathbb{R}_+,F)^{\otimes n}=L^2(\mathbb{R}^n_+,F^{\otimes n})$.  The full
Fock module over  $L^2(\mathbb{R}_+,F)$ is defined as
$$\mathcal{F}(L^2(\mathbb{R}_+,F))=\bigoplus_{n \in \mathbb{N}_0} L^2(\mathbb{R}_+,F)^{\otimes n},$$
where $L^2(\mathbb{R}_+,F)^{\otimes 0}=\B$. By $\omega$ we denote the vacuum, i. e. 1 in
$L^2(\mathbb{R}_+,F)^{\otimes 0}$.

By $\Delta_n$ we denote the indicator function of the subset $\{ (t_n,\dots ,t_1): t_n>\dots >t_1>0 \}$ of
$\mathbb{R}^n_+$. Clearly, $\Delta_n$ acts as a projection on $L^2(\mathbb{R}_+,F)$. Set $\Delta =
\bigoplus\limits_{n \in \mathbb{N}_0} \Delta_n$. The time ordered Fock module is the two-sided submodule
$$\G(F)=\Delta \mathcal{F}(L^2(\mathbb{R}_+,F))$$
of $L^2(\mathbb{R}_+,F)$. Denote by $\G_t(F)$ the restriction of $\G(F)$ to $[0,t)$. Setting
$$[u_{s,t}(F^m_s \otimes G^n_t)](s_m,\dots,s_1,t_n,\dots,t_1)=F^m_s(s_m-t,\dots,s_1-t)\otimes G^n_t(t_n,\dots,t_1)$$
bilinear unitaries $u_{s,t}:\G_{s}(F) \otimes \G_{t}(F) \to \G_{s+t}(F)$ turn the family $\G^{\otimes}(F)=
(\G_t(F))_{t \in \mathbb{R}_+}$ into a product system, the time ordered product system over $F$.

Obtained system admits a central unital unit $\omega=(\omega_t)$ with $\omega_t=\omega$, so it is a spatial
product system. Also, it can be shown that time ordered Fock module is type $I$ system \cite{JFA04}.

In \cite{KPR06}, Skeide proved that, up to isomorphisms, time ordered Fock modules are only examples of
spatial type $I$ product systems. In the same paper, he use this fact to define the index of spatial product
systems over $C^*$-algebra to be module $F$ that appears in Fock product system isomorphic to a subsystem of
a given spatial product system. The index of product systems over a $C^*$-algebra was defined in full
generality in \cite{KV} (and also in \cite{Biljana}) as a certain bimodule constructed from the maximal
continuous set of units. It presents a generalization of the index of product systems over $\mathbb{C}$
defined by Arveson \cite{Arv89}, as well as the index of a spatial product system of Hilbert modules defined
by Skeide \cite{KPR06}.

However, unlike Arveson case, a product subsystem of the time ordered poduct system (following \cite{KV} we
call it {\em subspatial}) need not be isomorphic to a time ordered product system, as it was shown in
\cite{PAMS10} by constructing a counterexample.

The point of this note is to calculate the index of this product subsystem, as it is defined in \cite{KV},
\cite{Biljana}. In other words we shall determine all units of such subsystem. We shall do that in Section
\ref{3}, whereas Section \ref{2} contains auxiliary statements, necessary for the proof of the main result.

\section{Preliminaries}\label{2}

In this Section we give a list of already known results.

\begin{proposition}\label{linear-comb} Let $E$ be a product system over a $C^*$-algebra $\B$, and let $u$, $u_1$, $\dots$,
$u_n\in\mathcal U$, where $\mathcal U$ is some maximal set of continuous units. Also, let $\beta\in\B$,
$\varkappa_j\in\mathcal B$, $j=1,\dots,n$ such that $\sum\varkappa_j=1$.

Then $\mathcal U$ contains units denoted by $u^\beta$, $\varkappa_1u_1\boxplus\dots\boxplus\varkappa_nu_n$
and $u_1\varkappa_1\boxplus\dots\boxplus u_n\varkappa_n$, which kernels are
\begin{equation}\label{stepenovanje beta}
\mathcal
L^{u^\beta,u^\beta}=\mathcal{L}^{u,u}+\beta^{*}\mathrm{id}_{\mathcal{B}}+\mathrm{id}_{\mathcal{B}}\beta,\qquad
\mathcal L^{u^\beta,\xi}=\mathcal{L}^{u,{\xi}}+{\beta}^{*}\mathrm{id}_{\mathcal{B}}.
\end{equation}
$$\L^{\boxplus\varkappa_ju_j,\boxplus\varkappa_ju_j}=\sum_{i,j=1}^n\varkappa_i^*\mathcal
L^{u_i,u_j}\varkappa_j,\qquad \L^{\boxplus\varkappa_ju_j,\xi}=\sum_{i=1}^n\varkappa_i^*\mathcal
L^{u_i,\xi},$$
$$\L^{\boxplus u_j\varkappa_j,\boxplus u_j\varkappa_j}=\sum_{i,j=1}^n\mathcal L^{u_i,u_j}L_{\varkappa_i^*}
R_{\varkappa_j},\qquad \L^{\boxplus u_j\varkappa_j,\xi}=\sum_{i=1}^n\mathcal L^{u_i,\xi} L_{\varkappa_i^*},$$
where $L_a,R_a:\mathcal{B}\rightarrow\mathcal{B}$ are the left and right multiplication operators for
$a\in\mathcal B$.
\end{proposition}

\begin{proof} This is \cite[Section 4.2, third example]{PAMS08} or \cite[Proposition 2.3]{KV}
\end{proof}

The set of all continuous units on some product system can be decomposed into mutually disjoint collection of
maximal continuous sets. This follows from \cite[Proposition 3.1]{KV} that is quoted here as

\begin{proposition}\label{maxcontset}
Let $\U$ denote the set of all continuous units on some product system $E$. The relation $\rho$ on $\U$
defined by
$$x \rho y \Leftrightarrow \{x,y\}\mbox{ is a continuous set}$$
is an equivalence relation. Consequently, the set of all continuous units on some product system can be
decomposed into mutually disjoint collection of maximal continuous sets.
\end{proposition}

\begin{proposition}\label{osobine skp}Let $E$ be a product system over  $\mathcal B$ with at least one continuous unit. (In view of
\cite[Definition 4.4]{SK03} this means that $E$ is non type $III$ product system.) Further, let $\mathcal
U=\mathcal U_\omega$ be the set of all uniformly continuous units that are equivalent to an arbitrary
continuous unit $\omega$ in $E$. (That refers to the equivalence relation $\rho$ on $\mathcal U$ defined in
Proposition \ref{maxcontset}.)

a) Define the addition and multiplication by $b\in\mathcal B$ on $\mathcal U_\omega$ by
\begin{equation}\label{operacije}
x+y=x\boxplus y\boxplus-\omega,\quad b\cdot x=bx\boxplus(1-b)\omega,\quad x\cdot b=xb\boxplus\omega(1-b),
\end{equation}
and define an equivalence relation $\approx$ by:
\begin{equation}\label{relekv}
x\approx y \Longleftrightarrow  x=y^\beta ,\ \beta\in\mathcal B.
\end{equation}
Then $\mathcal U_\omega$ has an algebraic structure of $\B-\B$ bimodule. Also, $\approx$ is compatible with
all algebraic operations in $\U_\omega$.

b) The mapping $\langle\ ,\ \rangle_b:\mathcal{U} \times \mathcal{U} \longrightarrow \mathcal{B}$ given by
\begin{equation}\label{skpdef}\langle
x,y\rangle_b=(\mathcal{L}^{x,y}-\mathcal{L}^{x,\omega}-\mathcal{L}^{\omega,y}+\mathcal{L}^{\omega,\omega})(b),
\end{equation}
where $\omega$ is the same as in (\ref{operacije}), is $\mathcal B$-valued semi-inner product (in the sense
that it can be degenerate). Properties of this mappings are
\begin{enumerate}
\item For all $x,y,z \in \mathcal{U}$, and $\alpha,\beta \in \mathbb{C}$ $\langle x,\alpha y+\beta
z \rangle_b=\alpha \langle x,y \rangle_b+\beta \langle x,z \rangle_b$;

\item For all $x,y \in \mathcal{U}$, $a \in \mathcal{B}$ $\langle x,y\cdot a \rangle_b=\langle x,y
\rangle_b a$;

\item For all $x,y \in \mathcal{U}$\ $\langle x,y \rangle_b=\langle y,x\rangle_b^{*}$;

\item For all $x \in \mathcal{U}$\ $\langle x,x\rangle_b\geq 0$;

\item If $x\approx x'$ and $y\approx y'$ then $\skp xy_b=\skp{x'}{y'}_b$.

\item\label{nej} If $0\le b(\in\mathcal B)\le1$ then for all $x\in\mathcal U$ we have $\skp xx_b\le\skp xx_1$.

\end{enumerate}

\end{proposition}

\begin{proof} This is \cite[Theorem 3.2, Proposition 3.3]{KV}.
\end{proof}

\begin{remark}Note that the kernels of $x+y$, $x\cdot a$, $a\cdot x$ are
\begin{equation}\label{sabiranje}\begin{gathered}
\mathcal
L^{x+y,x+y}=\mathcal{L}^{x,x}+\mathcal{L}^{x,y}-\mathcal{L}^{x,\omega}+\mathcal{L}^{y,x}+\mathcal{L}^{y,y}-\mathcal{L}^{y,\omega}-
\mathcal{L}^{\omega,x}-\mathcal{L}^{\omega,y}+\mathcal{L}^{\omega,\omega},\\
\mathcal L^{x+y,\xi}=\mathcal{L}^{x,\xi}+\mathcal{L}^{y,\xi}-\mathcal{L}^{\omega,\xi},
\end{gathered}\end{equation}
\begin{equation}\label{mnozenje}\begin{gathered}
\mathcal L^{x\cdot a,x\cdot a}=a^{*}\mathcal{L}^{x,x}a+(1-a)^{*}\mathcal{L}^{\omega,x} a+a^{*}
\mathcal{L}^{x,\omega}(1-a)+ (1-a)^{*}\mathcal{L}^{\omega,\omega}(1-a),\\
\mathcal L^{x\cdot a,\xi}=a^{*}\mathcal{L}^{x,\xi}+(1-a)^{*}\mathcal{L}^{\omega,\xi},\ \xi \in \mathcal{U},$$
\end{gathered}\end{equation}
\begin{equation}\label{mnozenje-levo}\begin{gathered}
\mathcal L^{a\cdot x,a\cdot x}=\mathcal{L}^{x,x}L_{a^*}R_a+\mathcal{L}^{\omega,x}L_{1-a^*}R_a+
\mathcal{L}^{x,\omega}L_{a^*}R_{1-a}+\mathcal{L}^{\omega,\omega}L_{1-a^*}R_{1-a},\\
\mathcal L^{a\cdot x,\xi}=\mathcal{L}^{x,\xi}L_{a^*}+\mathcal{L}^{\omega,\xi}L_{1-a^*},\ \xi \in
\mathcal{U},$$
\end{gathered}\end{equation}
where $L_b,R_b:\mathcal{B}\rightarrow\mathcal{B}$ are the left and right multiplication operators for
$b\in\mathcal B$.\end{remark}

\begin{definition} \cite[Definition 3.4]{KV} Index of a product system $E$ is the completion of pre-Hilbert left-right module
$\U/_{\sim}$, where $\sim$ is equivalence relation defined by $x \sim y$ if and only if $x-y \in N=\{x \in
\U|\ \skp{x}{x}_1=0\}$.
\end{definition}

\begin{proposition}
Let E be a subspatial product system.

(a) The equivalence relation $\approx$ from (\ref{relekv}) is characterized as follows:
$$x \approx y \Longleftrightarrow x\sim y.$$

(b) $\U/_{\sim}$ is complete in the norm induced by $\skp{\ }{\ }$, and therefore it is a Hilbert left-right
$\B-\B$ module.
\end{proposition}

\begin{proof}These are results from \cite[Proposition 5.5, Theorem 5.6]{KV}.
\end{proof}

\begin{proposition}\label{Fock-spatial}
Let $\G^{\otimes}(F)$ be the product system of time ordered Fock modules where $F$ is a two-sided Hilbert
module over $\B$.

a) All continuous units in $\G^{\otimes}(F)$ can be parameterized by the set $F\times\mathcal B$;

b) The unit that corresponds to pair $(\zeta,\beta)$ denote by $u(\zeta,\beta)$. The corresponding kernels
are given by
\begin{equation}\label{Fock kernel}
\L^{u(\zeta,\beta),u(\zeta',\beta')}(b)=\skp\zeta{b\zeta'}+\beta^*b+b\beta';
\end{equation}

c) We have $\mathcal U_{\G^{\otimes}(F)}/_\sim=\{[u(\zeta,0)]\;|\;\zeta\in F\}$. Also,
$\ind(\G^{\otimes}(F))$ is isomorphic to $F$ as Hilbert left-right module.
\end{proposition}

\begin{proof} a) These are results from \cite[Theorems 3 and 6]{LS01};

b) It is formula \cite[formula (3.5.2)]{JFA04};

c) By (\ref{Fock kernel}), we get
\begin{equation}\label{skpjed}
\skp{u(\zeta,\beta)}{u(\zeta',\beta')}_1=\skp\zeta{\zeta'},
\end{equation}
and the result follows immediately. See \cite[Example 6.2]{KV}.
\end{proof}

\begin{theorem}\label{podsistem}Let $\B=C_0[0,+\infty)+\mathbb{C} 1$ denote the unital $C^*$-algebra of all
continuous functions on $\mathbb{R}_+$ that have limit at infinity. Define the Hilbert $\B$-module $F:=\B$.
$F$ becomes a Hilbert $\B-\B$ module if we define the left action by
$$b\cdot x :=\s_1(b)x,$$
where $\s_t$ is the left shift by $t$, which acts as $[\s_t(b)](s)=b(s+t)$. $F^{\otimes n}$ is $\B$ as
Hilbert right module and with left action
$$b\cdot x :=\s_n(b)x.$$
Denote by $\xi=\xi(1,0)$ the unit of the time ordered product system over $F$ corresponding to the pair
$(1,0)$.

The product subsystem $E$ of $\G^{\otimes}(F)$ generated by the continuous unit $\xi$ has no central unit
vectors. In particular, it is not isomorphic to a time ordered system.
\end{theorem}

\begin{proof}Results of \cite[Section 3]{PAMS10}.
\end{proof}

\section{Result}\label{3}

In this section we calculate the index of the product system $E$ mentioned in Theorem \ref{podsistem}.

\begin{theorem} Let $E$ be the product system described in Theorem \ref{podsistem}. Then:

Among all continuous units $u=u(\zeta,\beta)$ in $\G^{\otimes}(F)$, exactly those for which $\zeta\in
1+C_0[0,+\infty)$ belongs to $E$.

The index of $E$ is isomorphic to $C_0[0,\infty)$ and it differs from the index of the containing time
ordered Fock space.
\end{theorem}

\begin{proof}
The product subsystem $E$ of $\G^{\otimes}(F)$ that is generated by $\xi$ is $E=(E_t)_{t \geq 0}$ with
$$E_t=\overline{span} \{ b_n\xi_{t_n} \otimes \dots \otimes b_1 \xi_{t_1} b_0|\  n \in \mathbb{N}, t_i>0, t_1+ \dots +t_n=t, b_i \in \B \},\ t>0.$$
Since $\xi=\xi(1,0)=(\xi_t)_{t \geq 0}$, $$\xi^0_t=1 \in \B,$$
$$\xi^n_t(r_n,\dots,r_1)=\underbrace{1 \otimes \dots \otimes 1}_{n} \in F^{\otimes n},\  t>r_n>\dots
>r_1>0.$$

{\sc $1^\circ$ step.} For $n \in \mathbb{N}$, let $\zeta \in F$ be the function that satisfies
$$\zeta(s) > 0,\  0 \leq s<n \mbox{\ \ and\ \ } \zeta(s)=1,\  s \geq n.$$
Let $\eta=\eta(\zeta,0)$ be the corresponding unit in $\G^{\otimes}(F)$. Let $b_0, b_1 \in \B$:
$$b_0(s)=\zeta(s) \zeta(s+1) \cdots \zeta(s+n-1),\ b_1(s)=\frac{1}{b_0(s)}.$$
For $n \in \mathbb{N}$ and $t>r_n>\cdots >r_1>0$,
$$(\eta^n_t(r_n,\cdots,r_1))(s)=(\underbrace{\zeta \otimes \cdots \otimes \zeta}_{n})(s)=\zeta(s+n-1)\cdots \zeta(s+1)\zeta(s),$$
and
$$((b_1 \xi_t b_0)^n(r_n,\dots,r_1))(s)=(b_1 \xi^n_t(r_n,\dots,r_1) b_0)(s)=(s_n(b_1)(\underbrace{1 \otimes \dots \otimes 1}_{n})b_0)(s)=$$
$$=b_1(s+n)b_0(s)=\frac{\zeta(s) \zeta(s+1) \dots \zeta(s+n-1)}{\zeta(s+n)\zeta(s+n+1)\dots \zeta(s+n+n-1)}=$$
$$=\zeta(s) \zeta(s+1) \dots \zeta(s+n-1).$$
Therefore, $\eta_t=b_1 \xi_t b_0$ and the unit $\eta$  belongs to the product subsystem $E=(E_t)_{t\geq 0}$.

{\sc $2^\circ$ step.} Let $n\in \mathbb{N}$ and consider $\zeta \in F$ such that $\zeta(s)=1,\  s \geq n$.
Let $\eta=\eta(\zeta,0)$ be the corresponding unit in $\G^{\otimes}(F)$. There exists $\alpha \in (0,1)$ such
that $\alpha \zeta(s)+(1-\alpha)1(s)>0$. Define  $\zeta'=\alpha \zeta+(1-\alpha)1 \in F$ and consider unit
$\eta'=\eta'(\zeta',0)$. From the previous part, $\eta'$ is a unit in $E$.

The mapping $t \mapsto \frac{1}{\alpha}\eta'_t+(1-\frac{1}{\alpha})\xi_t$ satisfies all assumptions of
Proposition \ref{linear-comb} and the resulting unit, denoted by $\theta$, belongs to $E$.

Let $u=u(\mu,\beta)$ be an arbitrary unit in $\G^{\otimes}(F)$. By Proposition \ref{linear-comb} and by
(\ref{Fock kernel}),
$$\L^{\theta,u}(b)=\frac{1}{\alpha} \L^{\eta',u}(b)+\left(1-\frac{1}{\alpha}\right) \L^{\xi,u}(b)=
\frac{1}{\alpha} \left( \skp{\zeta'}{b\mu}+b\beta\right)+\left(1-\frac{1}{\alpha}\right)
(\skp{1}{b\mu}+b\beta)=$$
$$=\skp{\frac{1}{\alpha} \zeta'+\left( 1-\frac{1}{\alpha} \right)1}{b\mu}+b\beta=\skp{\zeta}{b\mu}+b\beta=\L^{\eta,u}(b),\ b\in \B.$$
It follows $\L^{\theta,u}=\L^{\eta,u}$ and, by \cite[Lemma 2.6]{KV}
, $\eta=\theta$ belongs to $E$.

{\sc $3^\circ$ step.} Now consider $\zeta \in F$ such that $\lim\limits_{s\rightarrow +\infty} \zeta(s)=1$
and let $\eta=\eta(\zeta,0)$ be the corresponding unit in $\G^{\otimes}(F)$. There is a sequence $\zeta_n\in
F$ such that $\lim\limits_{n \rightarrow +\infty} \| \zeta- \zeta_n \|=0$, and $\zeta_n(s)=1$ for all $s\ge
n$. (For instance we can define $\zeta_n(s)=\zeta(s)/\zeta(n)$ for $s<n$ and some $n$ large enough that
$\zeta(n)\neq0$, and $\zeta_n(s)=1$ for $s\ge n$.)

Every unit $ \eta_n=\eta_n(\zeta_n,0)$ is a unit in product subsystem $E$ by the previous part. By
(\ref{skpjed}) it follows
\begin{multline*}\| \eta-\eta_n \|_{\U_{\G^{\otimes}(F)}}^2=\| \skp{\eta-\eta_n}{\eta-\eta_n}_1 \|=\\
\|\skp{\eta}{\eta}_1-\skp{\eta}{\eta_n}_1-\skp{\eta_n}{\eta}_1+\skp{\eta_n}{\eta_n}_1\|=\\
=\|\skp{\zeta}{\zeta}-\skp{\zeta}{\zeta_n}-\skp{\zeta_n}{\zeta}+\skp{\zeta_n}{\zeta_n}\|=\|\skp{\zeta-\zeta_n}{\zeta-\zeta_n}\|=
\|\zeta-\zeta_n\|^2.
\end{multline*}
This implies $\lim\limits_{n \rightarrow +\infty} \eta_n=\eta$.

It remains to show that $\eta \in \mathcal{U}_{E}$. Although it immediately follows from \cite[Theorem
5.6]{KV}, for the convenience of the reader we shall outline an independent proof.

For $0<\varepsilon\leq1$ there is $n_0 \in \mathbb{N}$ such that
\begin{equation}\label{Cauchy1}
\| \langle \eta_n-\eta, \eta_n-\eta \rangle \| < \varepsilon^2\quad\mbox{for }n \geq n_0.
\end{equation}
Let $b>0$ and $\tilde{b}=b/\|b\|$. For $n \geq n_0$, we have by Proposition \ref{osobine skp} (b.\ref{nej})
\begin{multline}\label{velika}\| \langle \eta_n,\eta_n \rangle_{\tilde{b}}- \langle \eta,\eta \rangle_{\tilde{b}} \| \leq \\\| \langle
\eta_n-\eta, \eta_n-\eta \rangle_{\tilde{b}}\|+\| \langle \eta_n-\eta,\eta \rangle_{\tilde{b}}\|+\| \langle
\eta,\eta_n-\eta
\rangle_{\tilde{b}}\| \leq\\
\leq \| \langle \eta_n-\eta, \eta_n-\eta \rangle_{1}\|+2\sqrt{\| \langle \eta_n-\eta, \eta_n-\eta
\rangle_{1}\|}\sqrt{\|
\langle \eta,\eta \rangle_1 \|} <\\
< \varepsilon^2+2\varepsilon \sqrt{\| \langle \eta,\eta \rangle_1 \|} < \varepsilon\mathrm{\ const}.
\end{multline}
Since $\L^{\omega,\eta_n}(\tilde{b})=\L^{\eta_n,\omega}(\tilde{b})=0$ by (\ref{Fock kernel}), we obtain
$$\mathcal{L}^{\eta_n,\eta_n}(\tilde{b})-\mathcal{L}^{\eta,\eta}(\tilde{b})=
\langle \eta_n,\eta_n \rangle_{\tilde{b}}-\langle \eta,\eta \rangle_{\tilde{b}}.$$
It follows, by (\ref{velika}),
\begin{equation}\label{conv}
\|(\mathcal{L}^{\eta_n,\eta_n}-\mathcal{L}^{\eta,\eta})(b)\|\leq \varepsilon\ \mathrm{const} \|b\|,
\end{equation}
for any $b\geq0$. Since every element of $\mathcal B$ is a linear combination of at most four positive
elements, we conclude that $\L^{\eta_n,\eta_n}$ converges to $\L^{\eta,\eta}$ in $B(\mathcal B)$, multiplying
the constant in (\ref{conv}) by $4$ if necessary.

For every unit $u$ in $ \G^{\otimes}(F)$, we have
$$\langle u,\eta_n-\eta \rangle_{\tilde{b}} \langle \eta_n-\eta,u \rangle_{\tilde{b}} \leq
\| \langle \eta_n-\eta, \eta_n-\eta \rangle_{\tilde{b}} \| \langle u,u \rangle_{\tilde{b}}.$$
By Proposition \ref{osobine skp} (b.\ref{nej}) and (\ref{Cauchy1}), $\| \langle \eta_n-\eta,u
\rangle_{\tilde{b}} \| < \varepsilon \sqrt{\| \langle u,u \rangle_1 \|}$ implying
\begin{equation}
\|(\mathcal{L}^{\eta_n,u}- \mathcal{L}^{\eta,u})(b) \| \leq \varepsilon \sqrt{\| \langle u,u \rangle_1 \|}
\|b\|,
\end{equation}
for all $\mathcal B\ni b\geq0$. As above we conclude that $\L^{\eta_n,u}$ converges to $\L^{\eta,u}$ in
$B(\mathcal B)$. The convergence is uniform with respect to $u$,\ $\|\skp{u}{u}\|\leq 1$.

Therefore, we proved
\begin{equation}
\lim\limits_{n \rightarrow {+\infty}} \| \mathcal{L}^{\eta_n,\eta_n}-\L^{\eta,\eta} \|=0,
\end{equation}
\begin{equation}
\lim\limits_{n \rightarrow {+\infty}} \| \mathcal{L}^{\eta_n,u}-\L^{\eta,u} \|=0.
\end{equation}

Since $\|\mathcal{L}^{\eta_n,\eta_n}\|,\ \|\mathcal{L}^{\eta_n,u}\|\leq\ \mathrm{const,\ for\ }n \in
\mathbb{N}$, the series
$$\sum\limits_{m=0}^{+\infty} \frac{t^m (\mathcal{L}^{\eta_n,\eta_n})^m}{m!}\quad\mbox{and}\quad
\sum\limits_{m=0}^{+\infty} \frac{t^m (\mathcal{L}^{\eta_n,u})^m}{m!}$$
uniformly converge with respect to $n\in\mathbb{N}$, which by Lebesgue dominant convergence theorem implies
$$\lim\limits_{n \rightarrow {+\infty}} \langle (\eta_n)_t,\bullet (\eta_n)_t \rangle=
\lim\limits_{n \rightarrow {+\infty}}e^{t \mathcal{L}^{\eta_n,\eta_n}}= \lim\limits_{n \rightarrow {+\infty}}
\sum\limits_{m=0}^{+\infty} \frac{t^m (\mathcal{L}^{\eta_n,\eta_n})^m}{m!}=e^{t \L^{\eta,\eta}},$$
$$\lim\limits_{n \rightarrow {+\infty}} \langle (\eta_n)_t,\bullet u_t \rangle=
\lim\limits_{n \rightarrow {+\infty}}e^{t \mathcal{L}^{\eta_n,u}}= \lim\limits_{n \rightarrow {+\infty}}
\sum\limits_{m=0}^{+\infty} \frac{t^m (\mathcal{L}^{\eta_n,u})^m}{m!}=e^{t \L^{\eta,u}}.$$
Therefore, by (\ref{e^L}),
$$\lim\limits_{n \rightarrow {+\infty}} \langle (\eta_n)_t,(\eta_n)_t \rangle=\langle \eta_t, \eta_t \rangle,$$
$$\lim\limits_{n \rightarrow {+\infty}} \langle (\eta_n)_t, \eta_t \rangle=\langle \eta_t, \eta_t \rangle.$$
Now, there holds
$$\lim\limits_{n\rightarrow +\infty} \|(\eta_n)_t-\eta_t\|^2=\lim\limits_{n\rightarrow +\infty} \|\langle(\eta_n)_t-\eta_t,(\eta_n)_t-\eta_t\rangle\|=0,$$
implying $\eta_t \in E_t$ and $\eta\in \U_{E}$.\\

Hence we proved that for every $\zeta \in 1+C_0[0,+\infty)$, a unit $\eta(\zeta,0) \in \U_{E}$.

{\sc $4^\circ$ step.} Also, it can be easily seen that if a unit $\eta(\zeta,0) \in \U_{\G^{\otimes}(F)}$
belongs to $\U_{E}$, there must be $\lim\limits_{s \rightarrow +\infty} \zeta(s)=1$, i. e. $\zeta \in
1+C_0[0,+\infty)$. This follows from the proof of \cite[Theorem 3.1]{PAMS10}.

Finally, $\U_{E}/_{\sim}=\{[u(\zeta,0)]\;|\;\zeta\in 1+C_0[0,+\infty)\}$ and $\mbox{ind}(E) \cong
1+C_0[0,+\infty)$. (The last isomorphism, as it is well known, relies on a translation that endow the affine
space $1+C_0[0,+\infty)$ with a structure of a bimodule. This structure is not canonical, it depends on the
choice of "zero", but all such linear structures are mutually isomorphic - again via a suitable translation.
cf.\ \cite[Proposition 5.4]{KV})
\end{proof}


\bibliographystyle{plain}

\bibliography{IndexSubspatial}

\begin{thebibliography}{10}

\bibitem{Arv89}
William Arveson.
\newblock Continuous analogues of {F}ock space.
\newblock Memoirs of the American Mathematical Society, vol 80/409, 1989.

\bibitem{Arv03}
William Arveson.
\newblock {\em Noncommutative Dynamics and ${E}$-Semigroups}.
\newblock Springer, New York, Berlin, Heidelberg, 2003.

\bibitem{JFA04}
S.~D. Barreto, B.~V.~R. Bhat, V.~Liebscher, and M.~Skeide.
\newblock Type ${I}$ product systems of {H}ilbert modules.
\newblock {\em J Funct. Anal.}, 212(1):121--181, 2004.

\bibitem{PAMS10}
B.~V.~R. Bhat, V.~Liebscher, and M.~Skeide.
\newblock Subsystems of {F}ock need not be {F}ock: Spatial ${CP}$-semigroups.
\newblock {\em Proc. Amer. Math. Soc.}, 138(7):2443--2456, 2010.

\bibitem{BS00}
B.~V.~R. Bhat and M.~Skeide.
\newblock {Tensor product systems of {H}ilbert modules and dilations of
  completely positive semigroups}.
\newblock {\em Infin.\ Dimens.\ Anal.\ Quantum Probab.\ Relat.\ Top.},
  03(4):519--575, 2000.

\bibitem{KV}
D.~J. Ke\v{c}ki\'c and B.~Vujo\v{s}evi\'c.
\newblock On the index of product systems of {H}ilbert modules.
\newblock {\em Filomat}, 29(5):1093--1111, 2015.

\bibitem{LS01}
V.~Liebscher and M.~Skeide.
\newblock Units for the time ordered {F}ock module.
\newblock {\em Infin.\ Dimens.\ Anal.\ Quantum Probab.\ Relat.\ Top.},
  04(4):545--551, 2001.

\bibitem{PAMS08}
V.~Liebscher and M.~Skeide.
\newblock Constructing units in product systems.
\newblock {\em Proc. Amer. Math. Soc.}, 136(3):989--997, 2008.

\bibitem{SK02}
M.~Skeide.
\newblock Dilations, product systems and weak dilations.
\newblock {\em Math. Notes}, 71(6):836--843, 2002.

\bibitem{SK03}
M.~Skeide.
\newblock Dilation theory and continuous tensor product systems of {H}ilbert
  modules.
\newblock In W.~Freudenberg, editor, {\em PQ–QP: Quantum Probability and White
  Noise Analysis XV}, Proceedings of the Conference Quantum Probability and
  Infinite Dimensional Analysis, pages 215--242. World Scientific, Singapore,
  2003.

\bibitem{KPR06}
M.~Skeide.
\newblock The index of (white) noises and their product systems.
\newblock {\em Infin.\ Dimens.\ Anal.\ Quantum Probab.\ Relat.\ Top.},
  9(04):617--655, 2006.

\bibitem{Biljana}
B.~Vujo\v{s}evi\'c.
\newblock The index of product systems of {H}ilbert modules: two equivalent
  definitions.
\newblock {\em Publ. Inst. Math. (Beograd) (N.S.)}, 97(111):49--56, 2015.

\end{thebibliography}


\end{document}